\documentclass{amsart}

\usepackage{lineno,hyperref}
\modulolinenumbers[5]

\usepackage{color}
\usepackage{amssymb}
\usepackage{amsmath}

\def\BibTeX{{\rm B\kern-.05em{\sc i\kern-.025em b}\kern-.08em
    T\kern-.1667em\lower.7ex\hbox{E}\kern-.125emX}}

\newtheorem{thm}{Theorem}[section]
\newtheorem{cor}[thm]{Corollary}
\newtheorem{lem}[thm]{Lemma}
\newtheorem{prop}[thm]{Proposition}

\newtheorem{rem}[thm]{Remark}
\newtheorem*{claim*}{Claim}
\theoremstyle{definition}

\numberwithin{equation}{section}

\usepackage{bbm}
\usepackage{mathrsfs}

\usepackage{pdflscape}

\begin{document}
\allowdisplaybreaks

\title[Moments and ergodicity of the JCIR process] % running head version
{Moments and ergodicity of the jump-diffusion CIR process}

\author{Peng Jin, Jonas Kremer and Barbara R\"udiger}
\address[Peng Jin]{Fakult\"at f\"ur Mathematik und Naturwissenschaften\\
        Bergische Universit\"at Wuppertal\\
        42119 Wuppertal, Germany}
%% Note the doubled @@:
\email[Peng Jin]{jin@uni-wuppertal.de}
\address[Jonas Kremer]{Fakult\"at f\"ur Mathematik und Naturwissenschaften\\
        Bergische Universit\"at Wuppertal\\
        42119 Wuppertal, Germany}
%% Note the doubled @@:
\email[Jonas Kremer]{j.kremer@uni-wuppertal.de}
\address[Barbara R\"udiger]{Fakult\"at f\"ur Mathematik und Naturwissenschaften\\
        Bergische Universit\"at Wuppertal\\
        42119 Wuppertal, Germany}
%% Note the doubled @@:
\email[Barbara R\"udiger]{ruediger@uni-wuppertal.de}

\date{\today}

\subjclass[2010]{Primary 60J25, 37A25; Secondary 60J35, 60J75}

\keywords{CIR model with jumps, ergodicity, exponential ergodicity, Forster-Lyapunov functions, fractional moments}

\maketitle

\begin{abstract}
We study the jump-diffusion CIR process, which is an extension of the  Cox-Ingersoll-Ross model and whose jumps are introduced by a subordinator. We provide sufficient conditions on the L\'evy measure of the subordinator under which the jump-diffusion CIR process is ergodic and exponentially ergodic, respectively. Furthermore, we characterize the existence of the $\kappa$-moment ($\kappa>0$) of the jump-diffusion CIR process by an integrability condition on the L\'evy measure of the subordinator.
\end{abstract}

%\linenumbers

\section{Introduction}
\label{intro}

In the present paper, we study the jump-diffusion CIR (shorted as
JCIR) process, which is an extension of the well-known Cox-Ingersoll-Ross
(shorted as CIR) model introduced in \cite{MR785475}. The JCIR process
$X=(X_{t})_{t\geqslant0}$ is defined as the unique strong solution
to the stochastic differential equation (SDE)
\begin{equation}
\mathrm{d}X_{t}=(a-bX_{t})\mathrm{d}t+\sigma\sqrt{X_{t}}\mathrm{d}B_{t}+\mathrm{d}J_{t},\quad t\geqslant0,\quad X_{0}\geqslant0\text{ a.s.},\label{eq:SDE JCIR}
\end{equation}
where $a\geqslant0$, $b>0$, $\sigma>0$ are constants, $(B_{t})_{t\geqslant0}$
is a one-dimensional Brownian motion and $(J_{t})_{t\geqslant0}$
is a pure jump L\'evy process with its L\'evy measure $\nu$ concentrating
on $(0,\infty)$ and satisfying
\begin{equation}
\int_{0}^{\infty}(z\wedge1)\nu(\mathrm{d}z)<\infty.\label{eq:levy measure property}
\end{equation}
We assume that $X_{0}$, $(B_{t})_{t\geqslant0}$ and $(J_{t})_{t\geqslant0}$
are independent. Note that the existence of a unique strong solution
to \eqref{eq:SDE JCIR} is guaranteed by \cite[Theorem 5.1]{MR2584896}.

The importance of the the CIR model and its extensions has been
demonstrated by their vast applications in mathematical finance, see,
e.g., \cite{MR785475,Duffie01,MR1850789,MR3343292}, and many others.
Since the CIR process is non-negative and mean-reverting, it is particularly
popular in interest rates and stochastic volatility modelling. These important features are inherited by
the JCIR process defined in \eqref{eq:SDE JCIR}. Moreover, compared to the CIR model, the JCIR process
has included possible jumps in it, which seems to make it a more appropriate
model to fit real world interest rates or volatility of asset prices.
As an  application of the JCIR process, Barletta and Nicolato
\cite{barletta2016orthogonal} recently studied a stochastic volatility
model with jumps for the sake of pricing of VIX options, where the volatility (or instantaneous variance process) of the asset
price process is modelled via the JCIR process.

An important issue for applications of the JCIR process is the estimation
of its parameters. To estimate the parameters of the original CIR
model, one can build some conditional least
square estimators (CLSEs) based on discrete observations (see Overbeck and Ryd\'en \cite{MR1455180}),
or some maximum likelihood estimators (MLEs) through continuous time
observations (see Overbeck \cite{MR1614256}, Ben Alaya and Kebaier \cite{MR2995525,MR3175784}). The parameter estimation
problem for the JCIR process is more complicated, since it has an
additional parameter $\nu$, which is the L\'evy measure of the driving noise $(J_{t})_{t\geqslant0}$ in \eqref{eq:SDE JCIR}
and thus an infinite dimensional object. Nevertheless, based on low frequency observations, Xu \cite{2013arXiv1310.4021X}
proposed some nonparametric estimators for $\nu$, given that $\nu$
is absolutely continuous with respect to the Lebesgue measure. Barczy \emph{et al.} \cite{BARCZY2017} studied also the maximum likelihood
estimator for the parameter $b$ of the JCIR process.

As seen in the aforementioned works \cite{BARCZY2017} and \cite{2013arXiv1310.4021X}, to study the fine properties of the estimators, a comprehension of the moments and long-time behavior
of the JCIR processes is required. In this paper we focus on these
two problems and analyze their subtle dependence on the big jumps of $(J_{t})_{t\geqslant0}$. Our first main result is
a characterization of the existence of moments of the JCIR process
in terms of the L\'evy measure $\nu$ of $(J_{t})_{t\geqslant0}$, namely,
we have the following:

\begin{thm}\label{thm:fractional moments} Consider the JCIR process $X=(X_{t})_{t\geqslant0}$ defined in \eqref{eq:SDE JCIR}. Let $\kappa>0$ be a constant.
Then the following three conditions are equivalent:
\begin{enumerate}
\item[(i)] $\mathbb{E}_x[X_t^{\kappa}]<\infty$ for all $x\in\mathbb{R}_{\geqslant0}$ and $t>0$,
\item[(ii)] $\mathbb{E}_x[X_t^{\kappa}]<\infty$ for some $x\in\mathbb{R}_{\geqslant0}$ and $t>0$,
\item[(iii)] $\int_{\lbrace z>1\rbrace}z^{\kappa}\nu(\mathrm{d}z)<\infty$,
\end{enumerate}
where the notation $\mathbb{E}_x[\cdot]$ means that the process $X$ considered under the expectation is with the initial condition $X_0=x$.
\end{thm}

After this paper was finished, we noticed that moments of general $1$-dimensional CBI processes were recently studied in \cite{2017arXiv170208698J}. If $\kappa\geqslant 1$ and $x>0$, our Theorem \ref{thm:fractional moments} can be viewed as a special case of \cite[Theorem 2.2]{2017arXiv170208698J}. However, to the authors' knowledge, the cases $0<\kappa<1$ and $\kappa\geqslant0$ with $x=0$ can not be handled by the approach used in \cite{2017arXiv170208698J}.

The second aim of this paper is to improve the results of \cite{MR3451177} on the ergodicity of the JCIR process. For a general time-homogeneous Markov process $M=(M_{t})_{t\geqslant0}$ with state space $E$, let $\mathbf{P}^{t}(x,\cdot):=\mathbb{P}_{x}\left(M_{t}\in\cdot\right)$ denote the distribution of $M_{t}$
with the initial condition $M_0=x\in E$. Following \cite{MR1234295}, we call $M$ \emph{ergodic} if it admits a unique invariant probability measure $\pi$ such that
\[
\lim_{t\to\infty}\left\Vert\mathbf{P}^t(x,\cdot)-\pi\right\Vert_{TV}=0,\quad \forall x\in E,
\]
where $\Vert\cdot\Vert_{TV}$ denotes the total variation norm for signed measures. The Markov process $M$ is called \emph{exponentially ergodic} if it is ergodic and in addition there exists a finite-valued function $B$ on $E$ and a positive constant $\delta$ such that
\[
\left\Vert\mathbf{P}^t(x,\cdot)-\pi\right\Vert_{TV}\leqslant B(x)e^{-\delta t},\quad \forall x\in E,\ t>0.
\]
Our second main result is the following:

\begin{thm}\label{thm:ergodicity_of_y_x}
Consider the JCIR process $(X_{t})_{t\geqslant0}$ defined by \eqref{eq:SDE JCIR} with parameters $a,b,\sigma$ and $\nu$, where $\nu$ is the L\'evy measure of $(J_{t})_{t\geqslant0}$. Assume $a>0$. We have:
\begin{enumerate}
\item[(a)] If $\int_{\lbrace z>1\rbrace}\log z\nu(\mathrm{d}z)<\infty$, then $X$ is ergodic.
\item[(b)] If $\int_{\lbrace z>1\rbrace}z^\kappa\nu(\mathrm{d}z)<\infty$ for some $\kappa>0$, then $X$ is exponentially ergodic.
\end{enumerate}
\end{thm}

We remark that similar results on the ergodicity of Ornstein-Uhlenbeck type processes were derived by Masuda, see \cite[Theorem 2.6]{MR2287102}. It is also worth mentioning that Jin \textit{et al.} \cite{MR3451177} already found a sufficient condition for the exponential ergodicity of the JCIR process, namely, if $a>0$, $\int_{\lbrace z\leqslant1\rbrace}z\log(1/z)\nu(\mathrm{d}z)<\infty$ and $\int_{\lbrace z>1\rbrace}z\nu(\mathrm{d}z)<\infty$. It is seen from part (b) of our Theorem \ref{thm:ergodicity_of_y_x} that these conditions can be significantly relaxed.

Our method to prove the (exponential) ergodicity of the JCIR process in question is based on the general theory of Meyn and Tweedie \cite{MR1174380,MR1234295} for ergodicity of Markov processes. As the first step, using a decomposition of its characteristic function, we show existence of positive transition densities of the JCIR process (see Proposition \ref{prop:existence and positivity of the density function}), which improves a similar result in \cite{MR3451177}. In the second step, we construct some Foster-Lyapunov functions for the JCIR process which enable us to prove the asserted (exponential) ergodicity by using the results in \cite{MR1174380,MR1234294,MR1234295}. For the construction of the Foster-Lyapunov functions we will use some ideas from \cite{MR2287102}.\\

The remainder of the article is organized as follows. In Section \ref{sec:prelim} we first introduce some notation and recall some basic facts on the JCIR process, then we establish an estimate for the moments of Bessel distributed random variables, which is crucial to Theorem \ref{thm:fractional moments}. In Section \ref{sec:lower bound of the transition densities} we will show that the JCIR process possesses positive transition densities. In Section \ref{sec:moments} we will prove Theorem \ref{thm:fractional moments}. Sections \ref{sec:ergodicity of the jump-diffusion CIR process} and \ref{sec:foster lyapunov} are devoted to the proof of Theorem \ref{thm:ergodicity_of_y_x}.

\section{Preliminaries and notation}
\label{sec:prelim}
\subsection{Notation}

Let $\mathbb{N}$, $\mathbb{Z}_{\geqslant0}$, $\mathbb{R}$, $\mathbb{R}_{\geqslant0}$
and $\mathbb{R}_{>0}$ denote the sets of positive integers, non-negative
integers, real numbers, non-negative real numbers and strictly positive
real numbers, respectively. Let $\mathbb{C}$ be the set of complex
numbers. We define the following subset of $\mathbb{C}$:
\[
\mathcal{U}:=\left\lbrace u\in\mathbb{C}\thinspace:\thinspace\mathrm{Re}\thinspace u\leqslant0\right\rbrace .
\]
We denote the Borel $\sigma$-algebra on $\mathbb{R}_{\geqslant0}$
simply by $\mathcal{B}(\mathbb{R}_{\geqslant0})$.

By $C^{2}(\mathbb{R}_{\geqslant0},\mathbb{R})$, and $C_{c}^{2}(\mathbb{R}_{\geqslant0},\mathbb{R})$
we denote the sets of $\mathbb{R}$-valued functions on $\mathbb{R}_{\geqslant0}$
that are twice continuously differentiable, and that are twice continuously
differentiable with compact support, respectively. For $a,b\in\mathbb{R}$,
we denote by $a\wedge b$ and $a\vee b$ the minimum and maximum of
$a$ and $b$, respectively. \\

We assume that $(\Omega,\mathcal{F},\left(\mathcal{F}_{t}\right)_{t\geqslant0},\mathbb{P})$
is a filtered probability space satisfying the usual conditions, i.e.,
$(\Omega,\mathcal{F},\mathbb{P})$ is complete, the filtration $\left(\mathcal{F}_{t}\right)_{t\geqslant0}$
is right-continuous and $\mathcal{F}_{0}$ contains all $\mathbb{P}$-null
sets in $\mathcal{F}$.

\subsection{The JCIR process}

Let $(B_{t})_{t\geqslant0}$ be a standard $(\mathcal{F}_{t})_{t\geqslant0}$-Brownian
motion and $(J_{t})_{t\geqslant0}$ be a $1$-dimensional $(\mathcal{F}_{t})_{t\geqslant0}$-L\'evy
process whose characteristic function is given by
\[
\mathbb{E}\left[e^{uJ_{t}}\right]=\exp\left\lbrace t\int_{0}^{\infty}\left(e^{uz}-1\right)\nu(\mathrm{d}z)\right\rbrace ,\quad(t,u)\in\mathbb{R}_{\geqslant0}\times\mathcal{U},
\]
where $\nu$ satisfies \eqref{eq:levy measure property}. We assume
that $(B_{t})_{t\geqslant0}$ and $(J_{t})_{t\geqslant0}$ are independent.
The L\'evy-It\^o representation of $(J_{t})_{t\geqslant0}$ takes the
form
\begin{equation}
J_{t}=\int_{0}^{t}\int_{0}^{\infty}zN(\mathrm{d}s,\mathrm{d}z),\quad t\geqslant0,\label{eq:levy ito decomposition of J_t}
\end{equation}
where $N(\mathrm{d}t,\mathrm{d}z)=\sum_{s\leqslant t}\delta_{(s,\Delta J_{s})}(\mathrm{d}t,\mathrm{d}z)$
is a Poisson random measure on $\mathbb{R}_{\geqslant0}$, where $\Delta J_{s}:=J_{s}-J_{s-}$,
$s>0$, $\Delta J_{0}:=0$, and $\delta_{(s,x)}$ denotes the Dirac
measure concentrated at $(s,x)\in\mathbb{R}_{\geqslant0}^{2}$.\\

It follows from \cite[Theorem 5.1]{MR2584896} that if $X_{0}$ is
independent of $(B_{t})_{t\geqslant0}$ and $(J_{t})_{t\geqslant0}$,
then there is a unique strong solution $(X_{t})_{t\geqslant0}$
to the SDE \eqref{eq:SDE JCIR}. Since the diffusion coefficient
in the SDE \eqref{eq:SDE JCIR} is degenerate at zero and only positive
jumps are possible, the JCIR process $(X_{t})_{t\geqslant0}$ stays
non-negative if $X_{0}\geqslant0$. This fact can be shown rigorously
with the help of comparison theorems for SDEs, for more details we
refer to \cite{MR2584896}. Using It\^o's formula, it is easy to see
that
\[
X_{t}=e^{-bt}\left(X_{0}+a\int_{0}^{t}e^{bs}\mathrm{d}s+\sigma\int_{0}^{t}e^{bs}\sqrt{X_{s}}\mathrm{d}B_{s}+\int_{0}^{t}e^{bs}\mathrm{d}J_{s}\right),\quad t\geqslant0.
\]
Moreover, the JCIR process $(X_{t})_{t\geqslant0}$ is a \emph{regular
affine process}, and the infinitesimal generator $\mathcal{A}$
of $X$ is given by
\begin{equation}
\left(\mathcal{A}f\right)(x)=(a-bx)\frac{\partial f(x)}{\partial x}+\frac{1}{2}\sigma^{2}x\frac{\partial^{2}f(x)}{\partial x^{2}}+\int_{0}^{\infty}\left(f(x+z)-f(x)\right)\nu(\mathrm{d}z),\label{eq:infinitesimal generator of JCIR}
\end{equation}
where $x\in\mathbb{R}_{\geqslant0}$ and $f\in C_{c}^{2}(\mathbb{R}_{\geqslant0},\mathbb{R})$.
If we write
\begin{align*}
\left(\mathcal{D}f\right)(x) & =(a-bx)\frac{\partial f(x)}{\partial x}+\frac{1}{2}\sigma^{2}x\frac{\partial^{2}f(x)}{\partial x^{2}},\\
\left(\mathcal{J}f\right)(x) & =\int_{0}^{\infty}\left(f(x+z)-f(x)\right)\nu(\mathrm{d}z),
\end{align*}
where $x\in\mathbb{R}_{\geqslant0}$ and $f\in C_{c}^{2}(\mathbb{R}_{\geqslant0},\mathbb{R})$,
we see that $\mathcal{A}f=\mathcal{D}f+\mathcal{J}f$.\\

\begin{rem} Let $a,b\in\mathbb{R}_{>0}$. If $\int_{\lbrace z>1\rbrace}\log z\nu(\mathrm{d}z)<\infty$,
then it follows from \cite[Theorem 3.16]{MR2390186} that the JCIR
process converges in law to a limit distribution $\pi$. Moreover,
as shown in \cite[p.80]{MR2779872}, the limit distribution $\pi$
is also the unique invariant distribution of the JCIR process. \end{rem}

Finally, we introduce some notation. Note that the strong solution
$(X_{t})_{t\geqslant0}$ of the SDE \eqref{eq:SDE JCIR} obviously
depends on its initial value $X_{0}$. From now on, we denote by $(X_{t}^{x})_{t\geqslant0}$
the JCIR process starting from a constant initial value $x\in\mathbb{R}_{\geqslant0}$,
i.e., $(X_{t}^{x})_{t\geqslant0}$ satisfies
\begin{equation}
\mathrm{d}X_{t}^{x}=(a-bX_{t}^{x})\mathrm{d}t+\sigma\sqrt{X_{t}^{x}}\mathrm{d}B_{t}+\mathrm{d}J_{t},\quad t\geqslant0,\quad X_{0}^{x}=x\in\mathbb{R}_{\geqslant0}.\label{eq:X^x_t}
\end{equation}

\subsection{Bessel distribution}
\label{subsec:bessel distr}
Suppose $\alpha$ and $\beta$ are positive constants. We call a probability
measure $m_{\alpha,\beta}$ on $(\mathbb{R}_{\geqslant0},\mathcal{B}(\mathbb{R}_{\geqslant0})$
a Bessel distribution with parameters $\alpha$ and $\beta$ if
\begin{equation}
m_{\alpha,\beta}(\mathrm{d}x):=e^{-\alpha}\delta_{0}(\mathrm{d}x)+\beta e^{-\alpha-\beta x}\sqrt{\alpha(\beta x)^{-1}}I_{1}\left(2\sqrt{\alpha\beta x}\right)\mathrm{d}x,\quad x\in\mathbb{R}_{\geqslant0},\label{eq: defi, m_a,b}
\end{equation}
where $\delta_{0}$ denotes the Dirac measure at the origin and $I_{1}$
is the modified Bessel function of the first kind, namely,
\begin{equation}
I_{1}(r)=\frac{r}{2}\sum_{k=0}^{\infty}\frac{\left(\frac{1}{4}r^{2}\right)^{k}}{k!(k+1)!},\quad r\in\mathbb{R}.\label{defi, bessel function}
\end{equation}
Let $\widehat{m}_{\alpha,\beta}(u):=\int_{\mathbb{R}_{\geqslant0}}\exp\lbrace ux\rbrace m_{\alpha,\beta}(\mathrm{d}x)$
for $u\in\mathcal{U}$ denote the characteristic function of the Bessel
distribution $m_{\alpha,\beta}$. It follows from \cite[p.291]{MR3451177}
that
\[
\widehat{m}_{\alpha,\beta}(u)=\exp\left\lbrace \frac{\alpha u}{\beta-u}\right\rbrace ,\quad u\in\mathcal{U}.
\]

To study the moments of the JCIR process, the lemma below plays a
substantial role.

\begin{lem}\label{lem:bound of bessel distribution} Let $\kappa>0$
and $\delta>0$ be positive constants. Then
\begin{enumerate}
\item[(i)] there exists a positive constant $C_{1}=C_{1}(\kappa)$ such that
for all $\alpha>0$ and $\beta>0$,
\[
\int_{\mathbb{R}_{\geqslant0}}x^{\kappa}m_{\alpha,\beta}(\mathrm{d}x)\leqslant C_{1}\frac{1+\alpha^{\kappa}}{\beta^{\kappa}}.
\]
\item[(ii)] there exists a positive constant $C_{2}=C_{2}(\kappa,\delta)$
such that for all $\alpha\geqslant\delta$ and $\beta>0$,
\[
\int_{\mathbb{R}_{\geqslant0}}x^{\kappa}m_{\alpha,\beta}(\mathrm{d}x)\geqslant C_{2}\frac{\alpha^{\kappa}}{\beta^{\kappa}}.
\]
\end{enumerate}
\end{lem}

\begin{proof} (i) If $0<\kappa\leqslant1$, then we can use Jensen's
inequality to obtain
\begin{equation}
\int_{\mathbb{R}_{\geqslant0}}x^{\kappa}m_{\alpha,\beta}(\mathrm{d}x)\leqslant\left(\int_{\mathbb{R}_{\geqslant0}}xm_{\alpha,\beta}(\mathrm{d}x)\right)^{\kappa}=\left(\frac{\alpha}{\beta}\right)^{\kappa},\label{eq1: Bessel proof}
\end{equation}
where the last identity holds because of
\[
\int_{\mathbb{R}_{\geqslant0}}xm_{\alpha,\beta}(\mathrm{d}x)=\left.\frac{\partial}{\partial u}\widehat{m}_{\alpha,\beta}(u)\right\vert _{u=0}=\frac{\alpha}{\beta}.
\]
For $\kappa=n\in\mathbb{N}$ with $n\geqslant2$, by \eqref{eq: defi, m_a,b}
and \eqref{defi, bessel function}, we have for all $\alpha,\beta>0$,
\begin{align}
\int_{\mathbb{R}_{\geqslant0}}x^{n}m_{\alpha,\beta}(\mathrm{d}x) & =\int_{\mathbb{R}_{\geqslant0}}x^{n}\left(e^{-\alpha}\delta_{0}(\mathrm{d}x)+\beta e^{-\alpha-\beta x}\sqrt{\alpha(\beta x)^{-1}}I_{1}\left(2\sqrt{\alpha\beta x}\right)\mathrm{d}x\right)\nonumber \\
 & =e^{-\alpha}\sum_{k=0}^{\infty}\frac{(\alpha\beta)^{k+1}}{k!(k+1)!}\int_{0}^{\infty}x^{n+k}e^{-\beta x}\mathrm{d}x\nonumber \\
 & =\frac{e^{-\alpha}}{\beta^{n}}\sum_{k=0}^{\infty}\frac{\alpha^{k+1}(n+k)!}{k!(k+1)!}\nonumber \\
 & =\frac{e^{-\alpha}}{\beta^{n}}\sum_{k=0}^{n-2}\frac{\alpha^{k+1}(n+k)!}{k!(k+1)!}\nonumber \\
 & \quad+\frac{e^{-\alpha}\alpha^{n}}{\beta^{n}}\sum_{k=n-1}^{\infty}\frac{\alpha^{k+1-n}}{(k+1-n)!}\cdot\frac{(k+1)\cdots(k+n)}{(k+2-n)\cdots(k+1)}.\label{eq2: Bessel proof}
\end{align}
Since
\[
\lim_{k\to\infty}\frac{(k+1)\cdots(k+n)}{(k+2-n)\cdots(k+1)}=1,
\]
it follows from \eqref{eq2: Bessel proof} that
\begin{align}
\int_{\mathbb{R}_{\geqslant0}}x^{n}m_{\alpha,\beta}(\mathrm{d}x) & \leqslant c_{1}\frac{e^{-\alpha}}{\beta^{n}}\left(\alpha+\alpha^{2}+\cdots+\alpha^{n-1}+\alpha^{n}\sum_{m=0}^{\infty}\frac{\alpha^{m}}{m!}\right)\nonumber \\
 & \leqslant c_{2}\left(\frac{1}{\beta^{n}}+\frac{\alpha^{n}}{\beta^{n}}\right),\quad\text{for all }\alpha,\beta>0,\label{eq3: Bessel proof}
\end{align}
where $c_{1}$ and $c_{2}$ are positive constants depending on $n$.

For the remaining possible $\kappa$, namely, $\kappa>1$ and $\kappa\notin\mathbb{N}$,
we can find $n\in\mathbb{N}$ and $\varepsilon\in(0,1]$ such that
$2\kappa=n+\varepsilon$. By \eqref{eq1: Bessel proof}, \eqref{eq3: Bessel proof}
and H\"older's inequality, we get for all $\alpha,\beta>0$,
\begin{align*}
\int_{\mathbb{R}_{\geqslant0}}x^{\kappa}m_{\alpha,\beta}(\mathrm{d}x) & \leqslant\left(\int_{\mathbb{R}_{\geqslant0}}x^{n}m_{\alpha,\beta}(\mathrm{d}x)\right)^{\frac{1}{2}}\left(\int_{\mathbb{R}_{\geqslant0}}x^{\varepsilon}m_{\alpha,\beta}(\mathrm{d}x)\right)^{\frac{1}{2}}\\
 & \leqslant c_{3}\left(\frac{1+\alpha^{n}}{\beta^{n}}\right)^{\frac{1}{2}}\left(\frac{\alpha}{\beta}\right)^{\frac{\varepsilon}{2}}\leqslant c_{4}\frac{\alpha^{\varepsilon/2}+\alpha^{(n+\varepsilon)/2}}{\beta^{(n+\varepsilon)/2}}\leqslant c_{5}\frac{1+\alpha^{\kappa}}{\beta^{\kappa}},
\end{align*}
where $c_{3}$, $c_{4}$ and $c_{5}$ are positive constants depending
on $\kappa$.

(ii) If $\kappa\geqslant1$, using again Jensen's inequality, we obtain
for all $\alpha,\beta>0$,
\[
\int_{\mathbb{R}_{\geqslant0}}x^{\kappa}m_{\alpha,\beta}(\mathrm{d}x)\geqslant\left(\int_{\mathbb{R}_{\geqslant0}}xm_{\alpha,\beta}(\mathrm{d}x)\right)^{\kappa}=\left(\frac{\alpha}{\beta}\right)^{\kappa}.
\]
Suppose now $0<\kappa<1$ and let $\theta:=1-\kappa\in(0,1)$. Consider
a random variable $\eta>0$ such that
\begin{equation}
\eta\sim\left(1-e^{-\alpha}\right)^{-1}\left(m_{\alpha,\beta}(\mathrm{d}x)-e^{-\alpha}\delta_{0}(\mathrm{d}x)\right).\label{eq4: Bessel proof}
\end{equation}
Then for $u\geqslant0$, we have
\begin{align*}
\mathbb{E}\left[e^{-u\eta}\right]
 & =\left(1-e^{-\alpha}\right)^{-1}\left(\widehat{m}_{\alpha,\beta}(-u)-e^{-\alpha}\right)\\
 & =\left(1-e^{-\alpha}\right)^{-1}\left(\exp\left\lbrace \frac{-\alpha u}{\beta+u}\right\rbrace -\exp\left\lbrace -\alpha\right\rbrace \right).
\end{align*}

Since, by the Fubini's theorem,
\begin{align*}
\int_{0}^{\infty}\frac{\partial}{\partial u}\mathbb{E}\left[e^{-u\eta}\right]u^{\theta-1}\mathrm{d}u & =-\int_{0}^{\infty}\mathbb{E}\left[Ye^{-u\eta}\right]u^{\theta-1}\mathrm{d}u\\
 & =-\mathbb{E}\left[\int_{0}^{\infty}\eta e^{-u\eta}u^{\theta-1}\mathrm{d}u\right]=-\mathbb{E}\left[\Gamma(\theta)\eta^{1-\theta}\right],
\end{align*}
it follows that
\begin{align}
\mathbb{E}\left[\eta^{\kappa}\right] & =\frac{-1}{\Gamma(\theta)}\int_{0}^{\infty}\frac{\partial}{\partial u}\mathbb{E}\left[e^{-u\eta}\right]u^{\theta-1}\mathrm{d}u\nonumber \\
 & =\frac{\alpha\beta}{\Gamma(\theta)\left(1-e^{-\alpha}\right)}\int_{0}^{\infty}\exp\left\lbrace \frac{-\alpha u}{\beta+u}\right\rbrace \frac{u^{\theta-1}}{(\beta+u)^{2}}\mathrm{d}u.\label{eq5: Bessel proof}
\end{align}
By \eqref{eq4: Bessel proof} and \eqref{eq5: Bessel proof}, we see
that
\[
\int_{\mathbb{R}_{\geqslant0}}x^{\kappa}m_{\alpha,\beta}(\mathrm{d}x)=\frac{\alpha\beta}{\Gamma(\theta)}\int_{0}^{\infty}\exp\left\lbrace \frac{-\alpha u}{\beta+u}\right\rbrace \frac{u^{\theta-1}}{(\beta+u)^{2}}\mathrm{d}u,\quad u\in\mathbb{R}_{\geqslant0}.
\]
By a change of variables $w:=\alpha u/\beta$, we get
\begin{align}
\int_{\mathbb{R}_{\geqslant0}}x^{\kappa}m_{\alpha,\beta}(\mathrm{d}x) & =\frac{\alpha\beta}{\Gamma(\theta)}\int_{0}^{\infty}\exp\left\lbrace -\alpha+\frac{\alpha\beta}{\beta+\frac{\beta w}{\alpha}}\right\rbrace \frac{\left(\frac{\beta w}{\alpha}\right)^{-\kappa}}{\left(\beta+\frac{\beta w}{\alpha}\right)^{2}}\frac{\beta}{\alpha}\mathrm{d}w\nonumber \\
 & =\frac{1}{\Gamma(\theta)}\left(\frac{\alpha}{\beta}\right)^{\kappa}\int_{0}^{\infty}\exp\left\lbrace \frac{-\alpha w}{\alpha+w}\right\rbrace \frac{w^{-\kappa}}{\left(1+w/\alpha\right)^{2}}\mathrm{d}w\nonumber \\
 & =:\frac{1}{\Gamma(\theta)}\left(\frac{\alpha}{\beta}\right)^{\kappa}I(\alpha).\label{eq6: Bessel proof}
\end{align}
By Fatou's lemma,
\begin{align*}
\liminf_{\alpha\to\infty}I(\alpha) & \geqslant\int_{0}^{\infty}\liminf_{\alpha\to\infty}\exp\left\lbrace \frac{-\alpha w}{\alpha+w}\right\rbrace \frac{w^{-\kappa}}{\left(1+w/\alpha\right)^{2}}\mathrm{d}w\\
 & =\int_{0}^{\infty}\exp\left\lbrace -w\right\rbrace w^{-\kappa}\mathrm{d}w=\Gamma(1-\kappa)>0.
\end{align*}
On the other hand, the function $(0,\infty)\ni\alpha\mapsto I(\alpha)$
is positive and continuous. So we can find a positive constant $c_{6}$
depending on $\kappa$ and $\delta$ such that $I(\alpha)\ge c_{6}$
for all $\alpha\in[\delta,\infty)$, which, together with \eqref{eq6: Bessel proof},
implies the assertion.
\end{proof}

\section{Positivity of the transition densities of the JCIR process}
\label{sec:lower bound of the transition densities}

The aim of this section is to prove that the JCIR process $X$ has
positive transition densities. Our approach is similar to that in
\cite[Proposition 4.5]{2016arXiv160706254J} and is based on the representation
of the law of $X_{t}^{x}$ as the convolution of two probability measures,
one of which is the distribution of the normal CIR process. Before
we prove the positivity of the transition densities, we recall the
characteristic function of $X_{t}^{x}$ and a decomposition of it,
established in \cite{MR3451177}.\\

Recall that $(X_{t}^{x})_{t\geqslant0}$ is given in \eqref{eq:X^x_t}.
Assume $a\in\mathbb{R}_{\geqslant0}$ and $b,\sigma\in\mathbb{R}_{>0}$.
Following \cite{MR3451177}, the characteristic function  of $X_{t}^{x}$
has the form
\begin{align}
\mathbb{E}\left[e^{uX_{t}^{x}}\right] & =\left(1-\tfrac{\sigma^{2}u}{2b}\left(1-e^{-bt}\right)\right)^{-\tfrac{2a}{\sigma^{2}}}\cdot\exp\left\lbrace x\psi(t,u)\right\rbrace \label{eq:characteristic function of the JCIR}\\
 & \quad\quad\quad\cdot\exp\left\lbrace \int_{0}^{t}\int_{0}^{\infty}\left(e^{z\psi(s,u)}-1\right)\nu(\mathrm{d}z)\mathrm{d}s\right\rbrace ,\quad(t,u)\in\mathbb{R}_{\geqslant0}\times\mathcal{U},\nonumber
\end{align}
where the function $\psi(t,u)$ is given by
\begin{equation}
\psi(t,u)=\frac{ue^{-bt}}{1-\tfrac{\sigma^{2}u}{2b}\left(1-e^{-bt}\right)}.\label{eq:psi}
\end{equation}

As mentioned in \cite{MR3451177}, the product of the first two terms
on the right-hand side of \eqref{eq:characteristic function of the JCIR}
is the characteristic function  of the CIR process. More precisely, consider
the unique strong solution $(Y_{t}^{x})_{t\geqslant0}$ of the following
stochastic differential equation \eqref{eq:SDE JCIR}
\begin{equation}
\mathrm{d}Y_{t}^{x}=(a-bY_{t}^{x})\mathrm{d}t+\sqrt{Y_{t}^{x}}\mathrm{d}B_{t},\quad t\geqslant0,\quad Y_{0}^{x}=x\in\mathbb{R}_{\geqslant0}\text{ a.s.}.\label{eq:SDE CIR}
\end{equation}
where $a\in\mathbb{R}_{\geqslant0}$, and $b,\sigma\in\mathbb{R}_{>0}$.
So $(Y_{t}^{x})_{t\geqslant0}$ is the CIR process starting from $x$.
Note that \eqref{eq:SDE CIR} is a special case of \eqref{eq:X^x_t}
with $J_{t}\equiv0$ (corresponding to $\nu=0$). By \eqref{eq:characteristic function of the JCIR},
we obtain
\begin{equation}
\mathbb{E}\left[e^{uY_{t}^{x}}\right]=\left(1-\tfrac{\sigma^{2}u}{2b}\left(1-e^{-bt}\right)\right)^{-\tfrac{2a}{\sigma^{2}}}\exp\left\lbrace \tfrac{xue^{-bt}}{1-\tfrac{\sigma^{2}u}{2b}\left(1-e^{-bt}\right)}\right\rbrace \label{eq:characteristic function of the CIR}
\end{equation}
for all $t\geqslant0$ and $u\in\mathcal{U}$.\\

We now turn to the third term on the right-hand side of \eqref{eq:characteristic function of the JCIR}.
Let $Z:=(Z_{t})_{t\geqslant0}$ be the unique strong solution of the
stochastic differential equation
\begin{equation}
\mathrm{d}Z_{t}=-bZ_{t}\mathrm{d}t+\sigma\sqrt{Z_{t}}\mathrm{d}B_{t}+\mathrm{d}J_{t},\quad t\geqslant0,\quad Z_{0}=0\text{ a.s.},\label{eq:SDE Z_t}
\end{equation}
where $\sigma\in\mathbb{R}_{>0}$. It is easy to see that \eqref{eq:SDE Z_t}
is also a special case of \eqref{eq:X^x_t} with $a=x=0$. Again by
\eqref{eq:characteristic function of the JCIR}, we have
\begin{equation}
\mathbb{E}\left[e^{uZ_{t}}\right]=\exp\left\lbrace \int_{0}^{t}\int_{0}^{\infty}\left(e^{z\psi(s,u)}-1\right)\nu(\mathrm{d}z)\mathrm{d}s\right\rbrace ,\quad(t,u)\in\mathbb{R}_{\geqslant0}\times\mathcal{U}.\label{eq:characteristic function of Z_t}
\end{equation}

It follows from \eqref{eq:characteristic function of the JCIR}, \eqref{eq:characteristic function of the CIR}
and \eqref{eq:characteristic function of Z_t} that
\[
\mathbb{E}\left[e^{uX_{t}^{x}}\right]=\mathbb{E}\left[e^{uY_{t}^{x}}\right]\mathbb{E}\left[e^{uZ_{t}}\right]
\]
for all $t\geqslant0$ and $u\in\mathcal{U}$. Let $\mu_{Y_{t}^{x}}$
and $\mu_{Z_{t}}$ be the probability laws of $Y_{t}^{x}$ and $Z_{t}$
induced on $(\mathbb{R}_{\geqslant0},\mathcal{B}(\mathbb{R}_{\geqslant0}))$,
respectively. Then the probability law $\mu_{X_{t}^{x}}$ of $X_{t}^{x}$
is given by
\begin{equation}
\mu_{X_{t}^{x}}=\mu_{Y_{t}^{y}}\ast\mu_{Z_{t}}, \label{eq: deco mu_X_t}
\end{equation}
where $\ast$ denotes the convolution of two measures.\\

\begin{prop}\label{prop:existence and positivity of the density function}
Assume $a>0$. For each $x\in\mathbb{R}_{\geqslant0}$ and $t\in\mathbb{R}_{>0}$,
the random variable $X_{t}^{x}$ possesses a density function $f_{X_{t}^{x}}(y)$,
$y\geqslant0$ with respect to the Lebesgue measure. Moreover, the
density function $f_{X_{t}^{x}}(y)$ is strictly positive for all
$y\in\mathbb{R}_{>0}$. \end{prop}

\begin{proof} According to \cite[Lemma 1]{MR3451177}, $X_{t}^{x}$
possesses a density function given by
\[
f_{X_{t}^{x}}(y)=\int_{\mathbb{R}_{\geqslant0}}f_{Y_{t}^{x}}(y-z)\mu_{Z_{t}}(\mathrm{d}z),\quad y\geqslant0,
\]
where $f_{Y_{t}^{x}}(y)$, $y\in\mathbb{R}$ denotes the density function
of $Y_{t}^{x}$, $t>0$. Since $(Y_{t}^{x})_{t\geqslant0}$ is the
CIR process, as well-known, we have $f_{Y_{t}^{x}}(y)>0$ for $y>0$
and $f_{Y_{t}^{x}}(y)\equiv0$ for $y<0$ (see, e.g., Cox \textit{et
al.} \cite[Formula (18)]{MR785475} or Jeanblanc \textit{et al.} \cite[Proposition 6.3.2.1]{MR2568861}
in case $x>0$ and Ikeda and Watanabe \cite[p.222]{MR1011252} in
case $x=0$). It remains to prove the strict positivity of $f_{X_{t}^{x}}(y)$
for all $y\in\mathbb{R}_{>0}$.

Let $t>0$ and $y>0$ be fixed. It follows that
\[
f_{X_{t}^{x}}(y)\geqslant\int_{[0,\delta]}f_{Y_{t}^{x}}(y-z)\mu_{Z_{t}}(\mathrm{d}z),
\]
where $\delta>0$ is small enough with $\delta<y$. Since $f_{Y_{t}^{x}}(y-z)>0$
for all $z\in[0,\delta]$, it is enough to check that $\mu_{Z_{t}}([0,\delta])>0$.
If $\mathbb{P}(Z_{t}=0)>0$, then we are done. So we now suppose
\begin{equation}
\mathbb{P}(Z_{t}=0)=0.\label{ass. : Z_t=00003D0}
\end{equation}
 Let
\[
\Delta_t(u)=\int_{0}^{t}\int_{0}^{\infty}\left(e^{z\psi(s,u)}-1\right)\nu(\mathrm{d}z)\mathrm{d}s,\quad u\in\mathcal{U},
\]
where $\psi$ is given in \eqref{eq:psi}. By \eqref{ass. : Z_t=00003D0},
we conclude
\begin{align}
\mathbb{E}\left[e^{u(Z_{t}-\delta)}\right] & -\mathbb{E}\left[e^{u(Z_{t}-\delta)}\mathbbm{1}_{\lbrace Z_{t}=0\rbrace}\right]\nonumber \\
 & =e^{-u\delta}\left(\mathbb{E}\left[e^{uZ_{t}}\right]-\mathbb{E}\left[e^{uZ_{t}}\mathbbm{1}_{\lbrace Z_{t}=0\rbrace}\right]\right)\nonumber \\
 & =e^{-u\delta}\left(e^{\Delta_t(u)}-\mathbb{P}\left(Z_{t}=0\right)\right)\nonumber \\
 & =e^{-u\delta/2}e^{\Delta_t(u)-u\delta/2}.\label{eq: lapl. Z_t-delta}
\end{align}
For all $u\in(-\infty,-1]$ and $s\in[0,t]$, we have
\begin{align}
\frac{\partial}{\partial u}\left(e^{z\psi(s,u)}-1\right) & =\frac{ze^{-bs}}{\left(1-\frac{\sigma^{2}u}{2b}\left(1-e^{-bs}\right)\right)^{2}}\exp\left\lbrace \frac{zue^{-bs}}{1-\frac{\sigma^{2}u}{2b}\left(1-e^{-bs}\right)}\right\rbrace \nonumber \\
 & \leqslant ze^{-bs}\mathbbm{1}_{\lbrace z\leqslant1\rbrace}+ze^{-bs}e^{-c_{1}z}\mathbbm{1}_{\lbrace z>1\rbrace}\leqslant c_{2}e^{-bs}(z\wedge1),\label{eq: domi for partial delta}
\end{align}
for some positive constants $c_{1}$ and $c_{2}$. By
the differentiation lemma \cite[Lemma 16.2]{MR1897176}, we see that $\Delta_t(u)$
is differentiable at $u\in(-\infty,-1]$ and
\begin{equation}
\frac{\partial}{\partial u}\left(\Delta_t(u)\right)=\int_{0}^{t}\int_{0}^{\infty}\frac{\partial}{\partial u}\left(e^{z\psi(s,u)}-1\right)\nu(\mathrm{d}z)\mathrm{d}s,\quad u\in(-\infty,-1].\label{eq: partial delta}
\end{equation}
Note that $\partial/(\partial u)(\exp\lbrace z\psi(s,u)\rbrace-1)>0$
for $z>0$, $u\in(-\infty,-1]$ and $s\in[0,t]$. Therefore, $\Delta_t(u)$
is strictly increasing in $u$ on $(-\infty,-1]$. Moreover, we have
\[
\lim_{u\to-\infty}\frac{\partial}{\partial u}\left(e^{z\psi(s,u)}-1\right)=\exp\left\lbrace \frac{-2bz}{\sigma^{2}\left(e^{bs}-1\right)}\right\rbrace \lim_{u\to-\infty}\frac{ze^{-bs}}{\left(1-\frac{\sigma^{2}u}{2b}\left(1-e^{-bs}\right)\right)^{2}}=0.
\]
By \eqref{eq: domi for partial delta}, \eqref{eq: partial delta}
and the Lebesgue dominated convergence theorem, $\partial/(\partial u)\Delta_t(u)\to0$
as $u\to-\infty$. So $\partial/(\partial u)(\Delta_t(u)-u\delta/2)\to-\delta/2$
as $u\to-\infty$, which implies that $\Delta_t(u)-u\delta/2$ is monotone
in $u$ for sufficiently small $u$ and thus
\begin{equation}
\lim_{u\to-\infty}e^{-u\delta/2}e^{\Delta_t(u)-u\delta/2}=\infty.\label{eq: limes delta}
\end{equation}
It follows from \eqref{eq: lapl. Z_t-delta} and \eqref{eq: limes delta}
that
\[
\lim_{u\to-\infty}\bigg(\mathbb{E}\left[e^{u(Z_{t}-\delta)}\right]-\mathbb{E}\left[e^{u(Z_{t}-\delta)}\mathbbm{1}_{\lbrace Z_{t}=0\rbrace}\right]\bigg)=\infty.
\]
Now, we must have $\mathbb{P}(Z_{t}\in(0,\delta])>0$, otherwise
\begin{align*}
\lim_{u\to-\infty}\bigg(\mathbb{E} & \left[e^{u(Z_{t}-\delta)}\right]-\mathbb{E}\left[e^{u(Z_{t}-\delta)}\mathbbm{1}_{\lbrace Z_{t}=0\rbrace}\right]\bigg)\\
 & \quad=\lim_{u\to-\infty}\left(\mathbb{E}\left[e^{u(Z_{t}-\delta)}\mathbbm{1}_{\lbrace0<Z_{t}\leqslant\delta\rbrace}\right]+\mathbb{E}\left[e^{u(Z_{t}-\delta)}\mathbbm{1}_{\lbrace Z_{t}>\delta\rbrace}\right]\right)=0.
\end{align*}
This completes the proof. \end{proof}

\section{Moments of the JCIR process}
\label{sec:moments}

In this section we prove Theorem \ref{thm:fractional moments}. Our approach is essentially motivated by the proof of \cite[Theorem 25.3]{MR3185174}.\\

\textit{Proof of Theorem} \ref{thm:fractional moments}. ``(iii)$\Rightarrow$(i)'': Let $\kappa>0$ be a constant. Suppose that $\int_{\lbrace z>1\rbrace}z^{\kappa}\nu(\mathrm{d}z)<\infty$.
Let $x\in\mathbb{R}_{\geqslant0}$ and $t>0$ be arbitrary. Note
that for all $(t,u)\in\mathbb{R}_{\geqslant0}\times\mathcal{U}$,
\begin{align}
\mathbb{E}\left[e^{uZ_{t}}\right] & =\exp\left\lbrace \int_{0}^{t}\int_{0}^{\infty}\left(e^{z\psi(s,u)}-1\right)\nu(\mathrm{d}z)\mathrm{d}s\right\rbrace \nonumber \\
 & =\exp\left\lbrace \int_{0}^{t}\int_{0}^{\infty}\left(e^{z\psi(s,u)}-1\right)\nu_{1}(\mathrm{d}z)\mathrm{d}s\right\rbrace \nonumber \\
 & \quad\thinspace\cdot\exp\left\lbrace \int_{0}^{t}\int_{0}^{\infty}\left(e^{z\psi(s,u)}-1\right)\nu_{2}(\mathrm{d}z)\mathrm{d}s\right\rbrace ,\label{eq: deco lap. Z_t}
\end{align}
where $\nu_{1}(\mathrm{d}z):=\mathbbm{1}_{\lbrace z\leqslant1\rbrace}\nu(\mathrm{d}z)$
and $\nu_{2}(\mathrm{d}z):=\mathbbm{1}_{\lbrace z>1\rbrace}\nu(\mathrm{d}z)$.
Similarly to \eqref{eq:SDE Z_t}, for $i=1,2$, we define $(Z_{t}^{i})_{t\geqslant0}$
as the unique strong solution of
\[
\mathrm{d}Z_{t}^{i}=-bZ_{t}^{i}\mathrm{d}t+\sigma\sqrt{Z_{t}^{i}}\mathrm{d}B_{t}+\mathrm{d}J_{t}^{i},\quad t\geqslant0,\quad Z_{0}^{i}=0\text{ a.s.},
\]
where $(J_{t}^{i})_{t\geqslant0}$ is a subordinator of pure jump-type
with L\'evy measure $\nu_{i}$. By \eqref{eq:characteristic function of Z_t},
we have
\begin{equation}
\mathbb{E}\left[e^{uZ_{t}^{i}}\right]=\exp\left\lbrace \int_{0}^{t}\int_{0}^{\infty}\left(e^{z\psi(s,u)}-1\right)\nu_{i}(\mathrm{d}z)\mathrm{d}s\right\rbrace ,\quad i=1,2,\ (t,u)\in\mathbb{R}_{\geqslant0}\times\mathcal{U}.\label{eq: lap. Z_t}
\end{equation}
It follows from \eqref{eq: deco lap. Z_t} and \eqref{eq: lap. Z_t}
that
\begin{equation}
\mu_{Z_{t}}=\mu_{Z_{t}^{1}}\ast\mu_{Z_{t}^{2}}.\label{eq: deco mu_Z_t}
\end{equation}

Let $f(y):=(|y|\vee1)^{\kappa}$, $y\in\mathbb{R}$. Then $f$ is
locally bounded and submultiplicative by \cite[Proposition 25.4]{MR3185174},
i.e., there exists a constant $c_{1}>0$ such that $f(y_{1}+y_{2})\leqslant c_{1}f(y_{1})f(y_{2})$
for all $y_{1},y_{2}\in\mathbb{R}$. Further, it is easy to see that
for any constant $c>0$, there exists a constant $c_{2}>0$ such that
$f(y)\leqslant c_{2}\exp\lbrace c|y|\rbrace$, $y\in\mathbb{R}$.
By \eqref{eq: deco mu_X_t} and \eqref{eq: deco mu_Z_t}, we get
\begin{align}
\mathbb{E}\left[f\left(X_{t}^{x}\right)\right] & \leqslant c_{1}^{2}\mathbb{E}\left[f\left(Y_{t}^{x}\right)\right]\mathbb{E}\left[f\left(Z_{t}^{1}\right)\right]\mathbb{E}\left[f\left(Z_{t}^{2}\right)\right]\nonumber \\
 & \leqslant c_{1}^{2}c_{2}\mathbb{E}\left[f\left(Y_{t}^{x}\right)\right]\mathbb{E}\left[e^{cZ_{t}^{1}}\right]\mathbb{E}\left[f\left(Z_{t}^{2}\right)\right].\label{esti: E f(X_t^x)}
\end{align}
By \cite[Proposition 3]{MR2995525}, we have $\mathbb{E}[f(Y_{t}^{x})]<\infty$.
The finiteness of the exponential moments of $Z_{t}^{1}$, i.e., $\mathbb{E}[\exp\lbrace cZ_{t}^{1}\rbrace]<\infty$,
follows by \cite[Theorem 2.14 (b)]{MR3313754}, since $(J_{t}^{1})_{t\geqslant0}$
has only small jumps.

We next show that $\mathbb{E}[f(Z_{t}^{2})]<\infty$. Note that $(J_{t}^{2})_{t\geqslant0}$
has only big jumps. By \cite[Lemma 2]{MR3451177}, we know that $Z_{t}^{2}$
is compound Poisson distributed, namely, we can find a probability
measure $\rho_{t}$ on $\mathbb{R}_{\geqslant0}$ such that
\[
\mathbb{E}\left[e^{uZ_{t}^{2}}\right]=e^{\lambda_{t}(\widehat{\rho}_{t}(u)-1)},\quad(t,u)\in\mathbb{R}_{>0}\times\mathcal{U},
\]
where $\lambda_{t}>0$ and $\widehat{\rho}_{t}$ denotes the characteristic
function of the measure $\rho_{t}$. More precisely, according to
\cite[see p.292]{MR3451177}, we have
\[
\rho_{t}=\lambda_{t}^{-1}\int_{0}^{t}\int_{\lbrace z>1\rbrace}m_{\alpha(z,s),\beta(z,s)}\nu(\mathrm{d}z)\mathrm{d}s,
\]
where $m_{\alpha(z,s),\beta(z,s)}$ is a Bessel distribution with
parameters $\alpha(z,s)$ and $\beta(z,s)$ given by
\[
\alpha(z,s):=\frac{2bz}{\sigma^{2}\left(e^{bs}-1\right)}\quad\text{and}\quad\beta(z,s):=\frac{2be^{bs}}{\sigma^{2}\left(e^{bs}-1\right)},
\]
and
\[
\lambda_{t}=\int_{0}^{t}\int_{\lbrace z>1\rbrace}\left(1-e^{-\alpha(z,s)}\right)\nu(\mathrm{d}z)\mathrm{d}s<\infty.
\]
By the Fubini's theorem, we obtain
\begin{equation}
\int_{\mathbb{R}_{\geqslant0}}f(y)\rho_{t}(\mathrm{d}y)=\lambda_{t}^{-1}\int_{0}^{t}\int_{\lbrace z>1\rbrace}\left(\int_{\mathbb{R}_{\geqslant0}}f(y)m_{\alpha(z,s),\beta(z,s)}(\mathrm{d}y)\right)\nu(\mathrm{d}z)\mathrm{d}s.\label{eq:formula for rho}
\end{equation}
By Lemma \ref{lem:bound of bessel distribution}, we have
\begin{align}
 & \int_{\mathbb{R}_{\geqslant0}}f(y)m_{\alpha(z,s),\beta(z,s)}(\mathrm{d}y)\leqslant\int_{\mathbb{R}_{\geqslant0}}(1+y^{\kappa})m_{\alpha(z,s),\beta(z,s)}(\mathrm{d}y)\nonumber \\
 & \ \leqslant1+C_{1}\frac{1+\alpha(z,s)^{\kappa}}{\beta(z,s)^{\kappa}}\leqslant1+C_{1}\sigma^{2\kappa}(2b)^{-\kappa}(1-e^{-bs})^{\kappa}+C_{1}e^{-\kappa bs}z^{\kappa}.\label{esti: fdm}
\end{align}
It follows from \eqref{eq:formula for rho} and \eqref{esti: fdm}
that
\begin{equation}
\int_{\mathbb{R}_{\geqslant0}}f(y)\rho_{t}(\mathrm{d}y)<\infty.\label{eq:integral with respect to rho is finite}
\end{equation}
Moreover, using \eqref{eq:integral with respect to rho is finite}
together with the submultiplicativity of $f$, we get
\begin{align}
\int_{\mathbb{R}_{\geqslant0}}f(y)\rho_{t}^{\ast n}(\mathrm{d}y) & =\int_{\mathbb{R}_{\geqslant0}}\cdots\int_{\mathbb{R}_{\geqslant0}}f(y_{1}+\cdots+y_{n})\rho_{t}(\mathrm{d}y_{1})\cdots\rho_{t}(\mathrm{d}y_{n})\nonumber \\
 & \leqslant c_{1}^{n}\left(\int_{\mathbb{R}_{\geqslant0}}f(y)\rho_{t}(\mathrm{d}y)\right)^{n}<\infty,\label{eq3: proof k-moment}
\end{align}
which implies
\begin{equation}
\mathbb{E}\left[f\left(Z_{t}^{2}\right)\right]=\int_{\mathbb{R}_{\geqslant0}}f(y)\mu_{Z_{t}^{2}}(\mathrm{d}y)=e^{-\lambda_{t}}\sum_{n=0}^{\infty}\frac{\lambda_{t}^{n}}{n!}\int_{\mathbb{R}_{\geqslant0}}f(y)\rho_{t}^{\ast n}(\mathrm{d}y)<\infty.\label{eq:finiteness if of kappa-moment of Z_t^2}
\end{equation}
By \eqref{esti: E f(X_t^x)} and \eqref{eq:finiteness if of kappa-moment of Z_t^2},
we obtain $\mathbb{E}\left[f\left(X_{t}^{x}\right)\right]<\infty$.
It follows easily that $\mathbb{E}\left[\left(X_{t}^{x}\right)^{\kappa}\right]<\infty$.

``(i)$\Rightarrow$(ii)'': It is clear.

``(ii)$\Rightarrow$(iii)'': Suppose now that $\mathbb{E}\left[\left(X_{t}^{x}\right)^{\kappa}\right]<\infty$
for some $x\in\mathbb{R}_{\geqslant0}$ and $t>0$. By \eqref{eq: deco mu_X_t},
we obtain
\[
\mathbb{E}\left[\left(X_{t}^{x}\right)^{\kappa}\right]=\int_{\mathbb{R}_{\geqslant0}}\int_{\mathbb{R}_{\geqslant0}}(y+z)^{\kappa}\mu_{Y_{t}^{x}}(\mathrm{d}y)\mu_{Z_{t}}(\mathrm{d}z)<\infty.
\]
So $\int_{\mathbb{R}_{\geqslant0}}(y+z)^{\kappa}\mu_{Z_{t}}(\mathrm{d}z)<\infty$
for some $y\in\mathbb{R}_{\geqslant0}$, which implies
\begin{equation}
\mathbb{E}\left[Z_{t}^{\kappa}\right]=\int_{\mathbb{R}_{\geqslant0}}z^{\kappa}\mu_{Z_{t}}(\mathrm{d}z)\leqslant\int_{\mathbb{R}_{\geqslant0}}(y+z)^{\kappa}\mu_{Z_{t}}(\mathrm{d}z)<\infty.\label{eq:finiteness of kappa moments of Z_t}
\end{equation}
Similarly, we can use \eqref{eq:finiteness of kappa moments of Z_t}
and \eqref{eq: deco mu_Z_t} to conclude that $(Z_{t}^{2})_{t\geqslant0}$
has finite moment of order $\kappa$. Let the function $f$ be as
above. Then $\mathbb{E}\left[f\left(Z_{t}^{2}\right)\right]\leqslant1+\mathbb{E}\left[\left(Z_{t}^{2}\right)^{\kappa}\right]<\infty$.
Since now all the summands in the last identity of \eqref{eq:finiteness if of kappa-moment of Z_t^2}
are finite, the summand corresponding to $n=1$ is also finite and
thus
\[
\int_{\mathbb{R}_{\geqslant0}}y^{\kappa}\rho_{t}(\mathrm{d}y)\leqslant\int_{\mathbb{R}_{\geqslant0}}f(y)\rho_{t}(\mathrm{d}y)<\infty.
\]
By the Fubini's theorem, we obtain
\begin{equation}
\int_{\mathbb{R}_{\geqslant0}}y^{\kappa}\rho_{t}(\mathrm{d}y)=\lambda_{t}^{-1}\int_{0}^{t}\int_{\lbrace z>1\rbrace}\left(\int_{\mathbb{R}_{\geqslant0}}y^{\kappa}m_{\alpha(z,s),\beta(z,s)}(\mathrm{d}y)\right)\nu(\mathrm{d}z)\mathrm{d}s<\infty.\label{eq1: proof k-moment}
\end{equation}
Noting that for all $s\in[0,t]$ and $z>1$,
\[
\alpha(z,s)=\frac{2bz}{\sigma^{2}\left(e^{bs}-1\right)}\geqslant\frac{2b}{\sigma^{2}\left(e^{bt}-1\right)}.
\]
By Lemma \ref{lem:bound of bessel distribution}, we can find a constant
$c_{3}=c_{3}(t)>0$ such that
\begin{equation}
\int_{\mathbb{R}_{\geqslant0}}y^{\kappa}m_{\alpha(z,s),\beta(z,s)}(\mathrm{d}y)\geqslant c_{3}\left(\frac{\alpha(z,s)}{\beta(z,s)}\right)^{\kappa}=c_{3}z^{\kappa}e^{-\kappa bs},\quad s\in[0,t],\ z>1.\label{eq2: proof k-moment}
\end{equation}
It follows from \eqref{eq1: proof k-moment} and \eqref{eq2: proof k-moment}
that $\int_{\lbrace z>1\rbrace}z^{\kappa}\nu(\mathrm{d}z)<\infty$.\qed\\

\begin{rem}
In Theorem \ref{thm:fractional moments} we have given a complete characterization of the existence of fractional moments for the JCIR process. For an explicit formula of integral moments of general CBI processes, the reader is referred to Barzy \textit{et al.} \cite{MR3540486}.
\end{rem}

Based on the proof of Theorem \ref{thm:fractional moments} we get
the following corollary.

\begin{cor}\label{prop:sup moments exist} Let $\kappa>0$ be a constant.
Suppose $\int_{\lbrace z>1\rbrace}z^{\kappa}\nu(\mathrm{d}z)<\infty$.
Then, for all $x\in\mathbb{R}_{\geqslant0}$ and $T>0$,
\[
\sup_{t\in[0,T]}\mathbb{E}_{x}\left[X_{t}^{\kappa}\right]<\infty.
\]
\end{cor}

\begin{proof} Let $f$, $Z_{t}^{1}$ and $Z_{t}^{2}$ be as in the
proof of Theorem \ref{thm:fractional moments}. Note that $|y|^{\kappa}\leqslant f(y)\le|y|^{\kappa}+1$
for all $y\in\mathbb{R}$. Since $\sup_{t\in\mathbb{R}_{\geqslant0}}\mathbb{E}[(Y_{t}^{x})^{\kappa}]<\infty$
due to \cite[Proposition 3]{MR2995525}, by \eqref{esti: E f(X_t^x)},
it suffices to check that
\[
\sup_{t\in[0,T]}\mathbb{E}\left[e^{cZ_{t}^{1}}\right]<\infty\quad\text{and}\quad\sup_{t\in[0,T]}\mathbb{E}\left[\left(Z_{t}^{2}\right)^{\kappa}\right]<\infty,\quad T>0,
\]
where $c>0$ is a constant to be chosen. It follows from \cite[Theorem 2.14 (b)]{MR3313754}
that
\[
\mathbb{E}\left[e^{cZ_{t}^{1}}\right]=\exp\left\lbrace \int_{0}^{t}\int_{0}^{1}\left(e^{z\psi(s,c)}-1\right)\nu_{1}(\mathrm{d}z)\mathrm{d}s\right\rbrace<\infty ,\quad c\in\mathbb{R},
\]
where $\psi$ is given in \eqref{eq:psi}. Now, we choose $c>0$ sufficiently
small such that $\psi(s,c)\geqslant0$ for all $s\in\mathbb{R}_{\geqslant0}$.
Hence, $\sup_{t\in[0,T]}\mathbb{E}[\exp\lbrace cZ_{t}^{1}\rbrace]\leqslant\mathbb{E}[\exp\lbrace cZ_T^1\rbrace]<\infty$. We next show that $\sup_{t\in[0,T]}\mathbb{E}\left[\left(Z_{t}^{2}\right)^{\kappa}\right]<\infty$.
By \eqref{eq:formula for rho}, \eqref{eq3: proof k-moment} and \eqref{eq:finiteness if of kappa-moment of Z_t^2},
we have for all $t\in[0,T],$
\begin{align}
\mathbb{E}\left[f\left(Z_{t}^{2}\right)\right] & \leqslant\exp\left\lbrace-\lambda_{t}+c_{1}\lambda_{t}\int_{\mathbb{R}_{\geqslant0}}f(y)\rho_{t}(\mathrm{d}y)\right\rbrace\nonumber \\
 & =\exp\left\lbrace-\lambda_{t}+c_{1}\int_{0}^{t}\int_{\lbrace z>1\rbrace}\left(\int_{\mathbb{R}_{\geqslant0}}f(y)m_{\alpha(z,s),\beta(z,s)}(\mathrm{d}y)\right)\nu(\mathrm{d}z)\mathrm{d}s\right\rbrace\nonumber \\
 & \leqslant\exp\left\lbrace c_{1}\int_{0}^{T}\int_{\lbrace z>1\rbrace}\left(\int_{\mathbb{R}_{\geqslant0}}f(y)m_{\alpha(z,s),\beta(z,s)}(\mathrm{d}y)\right)\nu(\mathrm{d}z)\mathrm{d}s\right\rbrace\nonumber \\
 & =\exp\left\lbrace c_{1}\lambda_{T}\int_{\mathbb{R}_{\geqslant0}}f(y)\rho_{T}(\mathrm{d}y)\right\rbrace.\label{eq: proof, coro, k-momonet}
\end{align}
It follows from \eqref{eq:integral with respect to rho is finite}
and \eqref{eq: proof, coro, k-momonet} that
\[
\sup_{t\in[0,T]}\mathbb{E}\left[\left(Z_{t}^{2}\right)^{\kappa}\right]
\leqslant\sup_{t\in[0,T]}\mathbb{E}\left[f\left(Z_{t}^{2}\right)\right]
\leqslant\exp\left\lbrace c_{1}\lambda_{T}\int_{\mathbb{R}_{\geqslant0}}f(y)\rho_{T}(\mathrm{d}y)\right\rbrace<\infty.
\]
This completes the proof. \end{proof}

\section{Ergodicity of the JCIR process}
\label{sec:ergodicity of the jump-diffusion CIR process}

In this section we prove the ergodicity of the JCIR process $X$ provided
that
\begin{equation}
\int_{\lbrace z>1\rbrace}\log z\nu(\mathrm{d}z)<\infty.\label{eq:integrability of log z}
\end{equation}
Our approach is based on the general theory of Meyn and Tweedie \cite{MR1234295}
for the ergodicity of Markov processes. The essential step is to find
a Foster-Lyapunov function in the sense of \cite[condition (CD2)]{MR1234295}.
In view of \eqref{eq:integrability of log z}, we choose the Foster-Lyapunov
function to be $V(x)=\log(1+x)$, $x\in\mathbb{R}_{\geqslant0}$.
We first show that this function $V$ is in the domain of the \textit{extended
generator} (see \cite[pp. 521-522]{MR1234295} for a definition) of
$X$.\\

\begin{lem}\label{lem:extended generator of the foster lyapunov function}
Suppose \eqref{eq:integrability of log z} is true. Let $V(x):=\log(1+x)$,
$x\in\mathbb{R}_{\geqslant0}$. Then for all $t>0$ and $x\in\mathbb{R}_{\geqslant0}$,
we have $\mathbb{E}_{x}\left[\int_{0}^{t}\left\vert \mathcal{A}V\left(X_{s}\right)\right\vert \mathrm{d}s\right]<\infty$
and
\begin{equation}
\mathbb{E}_{x}\left[V(X_{t})\right]=V(x)+\mathbb{E}_{x}\left[\int_{0}^{t}\mathcal{A}V\left(X_{s}\right)\mathrm{d}s\right],\label{eq:extended generator identity 1}
\end{equation}
where $\mathcal{A}$ is given in \eqref{eq:infinitesimal generator of JCIR}.
In other words, $V$ is in the domain of the extended
generator of $X$. \end{lem}

\begin{proof} It is easy to see that $V\in C^{2}(\mathbb{R}_{\geqslant0},\mathbb{R})$
and
\[
V'(x):=\frac{\partial}{\partial x}V(x)=(1+x)^{-1}\quad\text{and}\quad V''(x):=\frac{\partial^{2}}{\partial x^{2}}V(x)=-(1+x)^{-2}.
\]
Let $x\in\mathbb{R}_{\geqslant0}$ be fixed and assume that $X_{0}=x$
almost surely. In view of the L\'evy-It\^o decomposition of $\left(J_{t}\right)_{t\geqslant0}$
in \eqref{eq:levy ito decomposition of J_t}, we have
\[
X_{t}=x+\int_{0}^{t}(a-bX_{s})\mathrm{d}s+\sigma\int_{0}^{t}\sqrt{X_{s}}\mathrm{d}B_{s}+\int_{0}^{t}\int_{0}^{\infty}zN(\mathrm{d}s,\mathrm{d}z),\quad t\geqslant0,
\]
where $N(\mathrm{d}s,\mathrm{d}z)$ is defined in \eqref{eq:levy ito decomposition of J_t}.
By It\^o's formula, we obtain
\begin{align}
V(X_{t})-V(X_{0}) & =\int_{0}^{t}\left(a-bX_{s}\right)V'\left(X_{s}\right)\mathrm{d}s+\frac{\sigma^{2}}{2}\int_{0}^{t}X_{s}V''\left(X_{s}\right)\mathrm{d}s\nonumber \\
 & \quad+\sigma\int_{0}^{t}\sqrt{X_{s}}V'\left(X_{s}\right)\mathrm{d}B_{s}\nonumber \\
 & \quad+\int_{0}^{t}\int_{0}^{\infty}\left(V\left(X_{s-}+z\right)-V\left(X_{s-}\right)\right)N(\mathrm{d}s,\mathrm{d}z)\nonumber \\
 & =\int_{0}^{t}\left(a-bX_{s}\right)V'\left(X_{s}\right)\mathrm{d}s+\frac{\sigma^{2}}{2}\int_{0}^{t}X_{s}V''\left(X_{s}\right)\mathrm{d}s\nonumber \\
 & \quad+\int_{0}^{t}\int_{0}^{\infty}\left(V\left(X_{s-}+z\right)-V\left(X_{s-}\right)\right)\nu(\mathrm{d}z)\mathrm{d}s\nonumber \\
 & \quad+\sigma\int_{0}^{t}\sqrt{X_{s}}V'\left(X_{s}\right)\mathrm{d}B_{s}\nonumber \\
 & \quad+\int_{0}^{t}\int_{0}^{\infty}\left(V\left(X_{s-}+z\right)-V\left(X_{s-}\right)\right)\widetilde{N}(\mathrm{d}s,\mathrm{d}z)\nonumber \\
 & =\int_{0}^{t}(\mathcal{A}V)\left(X_{s}\right)\mathrm{d}s+M_{t}(V),\quad t\geqslant0,\label{eq:ito formula to g in the log case}
\end{align}
where $\widetilde{N}(\mathrm{d}s,\mathrm{d}z):=N(\mathrm{d}s,\mathrm{d}z)-\nu(\mathrm{d}z)\mathrm{d}s$
and
\begin{align*}
M_{t}(V) & :=\sigma\int_{0}^{t}\sqrt{X_{s}}V'\left(X_{s}\right)\mathrm{d}B_{s}\\
 & \quad+\int_{0}^{t}\int_{\lbrace z\leqslant1\rbrace}\left(V\left(X_{s-}+z\right)-V\left(X_{s-}\right)\right)\widetilde{N}(\mathrm{d}s,\mathrm{d}z)\\
 & \quad+\int_{0}^{t}\int_{\lbrace z>1\rbrace}\left(V\left(X_{s-}+z\right)-V\left(X_{s-}\right)\right)\widetilde{N}(\mathrm{d}s,\mathrm{d}z)\\
 & =D_{t}+J_{\ast,t}+J_{t}^{\ast}.
\end{align*}
Clearly, if $(M_{t}(V))_{t\geqslant0}$ is a martingale with respect
to the filtration $(\mathcal{F}_{t})_{t\geqslant0}$, by taking the
expectation of both sides of \eqref{eq:ito formula to g in the log case},
we see that condition \eqref{eq:extended generator identity 1} holds.\\

We start to prove that $(M_{t}(V))_{t\geqslant0}$ is a martingale
with respect to the filtration $(\mathcal{F}_{t})_{t\geqslant0}$.
Since
\[
\mathbb{E}_{x}\left[(D_{t})^{2}\right]=\sigma^{2}\int_{0}^{t}\mathbb{E}_{x}\left[X_{s}\left(1+X_{s}\right)^{-2}\right]\mathrm{d}s\leqslant\sigma^{2}\int_{0}^{t}\mathbb{E}_{x}\left[\left(1+X_{s}\right)^{-1}\right]\mathrm{d}s\leqslant t\sigma^{2}<\infty,
\]
it follows that $(D_{t})_{t\geqslant0}$ is a square-integrable martingale.
Note that
\begin{equation}
|V(y+z)-V(y)|\leqslant z\sup_{y\in\mathbb{R}_{\geqslant0}}\left\vert V'(y)\right\vert \leqslant z,\quad y,z\in\mathbb{R}_{\geqslant0}.\label{eq: mean value thm for V}
\end{equation}
Therefore,
\begin{align*}
\mathbb{E}_{x} & \left[\int_{0}^{t}\int_{\lbrace z\leqslant1\rbrace}\left(V\left(X_{s-}+z\right)-V\left(X_{s-}\right)\right)^{2}\nu(\mathrm{d}z)\mathrm{d}s\right]\leqslant t\int_{\lbrace z\leqslant1\rbrace}z^{2}\nu(\mathrm{d}z)<\infty,
\end{align*}
which implies that $(J_{\ast,t})_{t\geqslant0}$ is also a square-integrable
martingale by \cite[pp. 62, 63]{MR1011252}. If $y\mathbb{\in R}_{\geqslant0}$
and $z>1$, then
\begin{equation}
\left\vert V(y+z)-V(y)\right\vert =\log\left(1+\frac{z}{1+y}\right)\leqslant\log(1+z)\leqslant\log(2)+\log(z).\label{eq:estimate for V(x+z)-V(x) if z>1}
\end{equation}
So
\begin{align*}
\mathbb{E}_{x} & \left[\int_{0}^{t}\int_{\lbrace z>1\rbrace}\left\vert V(X_{s-}+z)-V(X_{s-})\right\vert \nu(\mathrm{d}z)\mathrm{d}s\right]\\
 & \quad\leqslant t\int_{\lbrace z>1\rbrace}\left(\log(2)+\log(z)\right)\nu(\mathrm{d}z)\\
 & \quad=t\log(2)\nu(\lbrace z>1\rbrace)+t\int_{\lbrace z>1\rbrace}\log(z)\nu(\mathrm{d}z)<\infty,\quad t\geqslant0,
\end{align*}
and hence, by \cite[Lemma 3.1 and p. 62]{MR1011252}, $(J_{t}^{\ast})_{t\geqslant0}$
is a martingale. Consequently, $(M_{t}(V))_{t\geqslant0}=(D_{t}+J_{\ast,t}+J_{t}^{\ast})_{t\geqslant0}$
is a martingale with respect to the filtration $(\mathcal{F}_{t})_{t\geqslant0}$.\\

Next, we show that $\mathbb{E}_{x}\left[\int_{0}^{t}\left\vert \mathcal{A}V\left(X_{s}\right)\right\vert \mathrm{d}s\right]<\infty$
for all $t\geqslant0$. By the decomposition of $\mathcal{A}$ into
a \emph{diffusion part} $\mathcal{D}$ and a \emph{jump part} $\mathcal{J}$
as introduced in Section \ref{sec:prelim}, we can write $\mathcal{A}V=\mathcal{D}V+\mathcal{J}V$.
Concerning the diffusion part $\mathcal{D}V$, it is easy to see that
\begin{equation}
\sup_{y\mathbb{\in R}_{\geqslant0}}\left\vert (\mathcal{D}V)(y)\right\vert =\sup_{y\mathbb{\in R}_{\geqslant0}}\left\vert (a-by)(1+y)^{-1}-\frac{\sigma^{2}}{2}y(1+y)^{-2}\right\vert <\infty.\label{esti: DV, log}
\end{equation}
For the jump part $\mathcal{J}V$, we decompose it further as $\mathcal{J}V=\mathcal{J}_{\ast}V+\mathcal{J}^{\ast}V$,
where
\begin{align}
(\mathcal{J}_{\ast}V)(y) & =\int_{\lbrace z\leqslant1\rbrace}\left(V(y+z)-V(y)\right)\nu(\mathrm{d}z),\label{eq:small jump part}\\
(\mathcal{J}^{\ast}V)(y) & =\int_{\lbrace z>1\rbrace}\left(V(y+z)-V(y)\right)\nu(\mathrm{d}z).\label{eq:big jump part}
\end{align}
By \eqref{eq: mean value thm for V}, we have
\begin{equation}
\left\vert (\mathcal{J}_{\ast}V)(y)\right\vert \leqslant\int_{\lbrace z\leqslant1\rbrace}z\nu(\mathrm{d}z)<\infty,\quad y\in\mathbb{R}_{\geqslant0}.\label{eq:finiteness of small jumps}
\end{equation}
Concerning $\mathcal{J}^{\ast}$, it follows from \eqref{eq:estimate for V(x+z)-V(x) if z>1}
that
\begin{equation}
\left\vert (\mathcal{J}^{\ast}V)(y)\right\vert \leqslant\log(2)\nu(\lbrace z>1\rbrace)+\int_{\lbrace z>1\rbrace}\log z\nu(\mathrm{d}z)<\infty,\quad y\in\mathbb{R}_{\geqslant0}.\label{eq:finiteness of big jumps}
\end{equation}
Combining \eqref{esti: DV, log}, \eqref{eq:finiteness of small jumps}
and \eqref{eq:finiteness of big jumps} yields that $\vert\mathcal{A}V\vert$
is bounded on $\mathbb{R}_{\geqslant0}$, which implies $\mathbb{E}_{x}\left[\int_{0}^{t}\left\vert \mathcal{A}V\left(X_{s}\right)\right\vert \mathrm{d}s\right]<\infty$
for all $t\geqslant0$. \end{proof}

We are ready to prove the ergodicity of the JCIR process $(X_{t})_{t\geqslant0}$
under \eqref{eq:integrability of log z}.\\

\textit{Proof of Theorem} \ref{thm:ergodicity_of_y_x} (a).
In view of \cite[Theorem 5.1]{MR1234295}, to prove the ergodicity
of the JCIR process $(X_{t})_{t\geqslant0}$, it is enough to check
that
\begin{enumerate}
\item[(i)] $(X_{t})_{t\geqslant0}$ is a non-explosive (Borel) right process
(see, e.g., \cite[p.38]{MR958914} or \cite[p.67]{MR2250510} for
a definition of a (Borel) right process);
\item[(ii)] all compact sets of the state space $\mathbb{R}_{\geqslant0}$ are
petite for some skeleton chain (see \cite[p.500]{MR1234294} for a definition);
\item[(iii)] there exist positive constants $c$, $M$ such that
\begin{equation}
(\mathcal{A}V)(x)\leqslant-c+M\mathbbm{1}_{K}(x),\quad x\in\mathbb{R}_{\geqslant0},\label{eq:ergodicity generator condition}
\end{equation}
for some compact subset $K\subset\mathbb{R}_{\geqslant0}$, where
$V(x)=\log(1+x)$, $x\in\mathbb{R}_{\geqslant0}$.
\end{enumerate}
We proceed to prove (i)-(iii).\\

In view of \cite[Corollary 4.1.4]{MR2250510}, $(X_{t})_{t\geqslant0}$
is a right process, since it possesses the Feller property as an affine
process (see \cite[Theorem 2.7]{MR1994043}).\\

According to Proposition \ref{prop:existence and positivity of the density function},
we can proceed in the very same way as in Jin \textit{et al.} \cite[Theorem 1]{MR3451177}
to see that for each $n\in\mathbb{Z}_{\geqslant0}$ the $\delta$-skeleton chain $X_{n\delta}$, $\delta>0$ being a constant, is
irreducible with respect to the Lebesgue measure on $\mathbb{R}_{\geqslant0}$.
Since $(X_{t})_{t\geqslant0}$ has the Feller property, the claim
(ii) now follows from \cite[Proposition 6.2.8]{MR2509253}.\\

Finally, we prove (iii). As shown in the proof of Lemma \ref{lem:extended generator of the foster lyapunov function},
$\vert\mathcal{A}V\vert$ is bounded on $\mathbb{R}_{\geqslant0}$.
Therefore, to get \eqref{eq:ergodicity generator condition}, it suffices
to show that $\lim_{x\to\infty}\mathcal{A}V(x)$ exists and is negative.
As before, we write $\mathcal{A}V=\mathcal{D}V+\mathcal{J}V$. It
is easy to see that
\begin{align*}
\lim_{x\to\infty}(\mathcal{D}V)(x)=\lim_{x\to\infty}\left[(a-bx)(1+x)^{-1}-\frac{\sigma^{2}}{2}x(1+x)^{-2}\right]=-b.
\end{align*}
Next, we consider the jump part $\mathcal{J}V$. Note that
\[
V(x+z)-V(x)=\log\left(1+\frac{z}{1+x}\right)\longrightarrow0\quad\text{as }x\to\infty.
\]
On the other hand, by \eqref{eq: mean value thm for V} and \eqref{eq:estimate for V(x+z)-V(x) if z>1},
we have

\[
|V(x+z)-V(x)|\leqslant z\mathbbm{1}_{\lbrace z\leqslant1\rbrace}+\left[\log(2)+\log(z)\right]\mathbbm{1}_{\lbrace z>1\rbrace},
\]
where the function on the right-hand side is integrable with respect
to $\nu$. By the dominated convergence theorem, we obtain $\lim_{x\to\infty}(\mathcal{J}V)(x)=0$.
This completes the proof.\qed

\begin{rem} According to the discussion after \cite[Proposition 2.5]{MR663900},
a direct but important consequence of our ergodic result is the following:
under the assumptions of Theorem \textup{\ref{thm:ergodicity_of_y_x} (a)},
for all Borel measurable functions $f:\mathbb{R}_{\geqslant0}\to\mathbb{R}$
with $\int_{\mathbb{R}_{\geqslant0}}\vert f(x)\vert\pi(\mathrm{d}x)<\infty$,
it holds
\begin{equation}
\mathbb{P}\left(\lim_{T\to\infty}\frac{1}{T}\int_{0}^{T}f(X_{s})\mathrm{d}s=\int_{\mathbb{R}_{\geqslant0}}f(x)\pi(\mathrm{d}x)\right)=1.\label{eq: law of large numbers}
\end{equation}
The convergence \eqref{eq: law of large numbers} may be very useful
for parameter estimation of the JCIR process. \end{rem}

\section{Exponential ergodicity of the JCIR process}
\label{sec:foster lyapunov}

Our aim of this section is to show that the JCIR process $X$ is exponentially
ergodic if
\begin{equation}
\int_{\{z>1\}}z^{\kappa}\nu(dz)<\infty\quad\text{for some }\kappa>0\text{.}\label{conditon: kappa moment}
\end{equation}

As in previous works (see, e.g., \cite{MR3451177} and \cite{2016arXiv160706254J}) the following proposition will play an essential role in proving exponential ergodicity of the JCIR process $X$, provided that \eqref{conditon: kappa moment} holds.

\begin{prop}\label{lem:foster_lyapunov_fct} Suppose \eqref{conditon: kappa moment} is true. Let $V\in C^{2}(\mathbb{R}_{\geqslant0},\mathbb{R})$ be
nonnegative and such that $V(x)=x^{\kappa\wedge1}$ for $x\geqslant1$.
Then there exist positive constants $c,M$ such that
\begin{equation}
\mathbb{E}_{x}\left[V(X_{t})\right]\leqslant e^{-ct}V(x)+\frac{M}{c}\label{eq:Foster-Lyapunov 1}
\end{equation}
for all $(t,x)\in\mathbb{R}_{\geqslant0}^{2}$. \end{prop}

\begin{proof} If $\kappa\geqslant1$, then it follows from \eqref{conditon: kappa moment}
that $\int_{\{z>1\}}z\nu(dz)<\infty$, which, together with \cite[Lemma 3]{MR3451177},
implies
\[
\mathbb{E}_{x}\left[X_{t}\right]\leqslant xe^{-bt}+M_{1},\quad t>0,x\geqslant0,
\]
for some constant $0<M_{1}<\infty$. In this case, we have
\begin{align*}
\mathbb{E}_{x}\left[V(X_{t})\right] & =\mathbb{E}_{x}\left[V(X_{t})\mathbbm{1}_{\lbrace X_{t}>1\rbrace}\right]+\mathbb{E}_{x}\left[V(X_{t})\mathbbm{1}_{\lbrace X_{t}\leqslant1\rbrace}\right]\\
 & \leqslant\mathbb{E}_{x}\left[X_{t}\right]+\sup_{y\in[0,1]}|V(y)|\\
 & \leqslant xe^{-bt}+M_{1}+\sup_{y\in[0,1]}|V(y)|\\
 & \leqslant\left(V(x)+1\right)e^{-bt}+M_{1}+\sup_{y\in[0,1]}|V(y)|\\
 & \leqslant V(x)e^{-bt}+M_{2},
\end{align*}
where $M_{2}:=1+M_{1}+\sup_{y\in[0,1]}|V(y)|<\infty$ is a constant.
Hence \eqref{eq:Foster-Lyapunov 1} is true when $\kappa\geqslant1$.
So in the following we assume $0<\kappa<1$.

Define $g(t,x):=\exp(ct)V(x)$, where $c\in\mathbb{R}_{>0}$ is a
constant to be determined later. Then,
\begin{align*}
g_{t}'(t,x) & :=\frac{\partial}{\partial t}g(t,x)=ce^{ct}V(x),\\
g_{x}'(t,x) & :=\frac{\partial}{\partial x}g(t,x)=\begin{cases}
\kappa e^{ct}x^{\kappa-1}, & x>1,\\
e^{ct}V'(x), & x\in[0,1],
\end{cases}\\
g_{x}''(t,x) & :=\frac{\partial^{2}}{\partial x^{2}}g(t,x)=\begin{cases}
\kappa(\kappa-1)e^{ct}x^{\kappa-2}, & x>1,\\
e^{ct}V''(x), & x\in[0,1].
\end{cases}
\end{align*}
Applying It\^o's formula for $g(t,X_{t})$, we obtain
\begin{equation}
g(t,X_{t})-g(0,X_{0})=\int_{0}^{t}(\mathcal{L}g)(s,X_{s})\mathrm{d}s+\int_{0}^{t}g_{s}'(s,X_{s})\mathrm{d}s+M_{t}(g),\quad t\geqslant0,\label{eq:ito formula to g}
\end{equation}
where the operator $\mathcal{L}$ is given by $(\mathcal{L}g)(s,X_{s})=\exp\lbrace cs\rbrace(\mathcal{A}V)(X_{s})$
with $\mathcal{A}$ as in \eqref{eq:infinitesimal generator of JCIR}
and
\begin{align*}
M_{t}(g) & :=\sigma\int_{0}^{t}\sqrt{X_{s}}g_{x}'(s,X_{s})\mathrm{d}B_{s}+\int_{0}^{t}\int_{0}^{\infty}\left(g(s,X_{s-}+z)-g(s,X_{s-})\right)\widetilde{N}(\mathrm{d}s,\mathrm{d}z)\\
 & \thickspace=G_{t}(g)+J_{t}(g),\quad\text{for all }t\geqslant0.
\end{align*}
We will complete the proof in three steps.\\

``\emph{Step 1}'': We check that $(M_{t}(g))_{t\geqslant0}$ is
a martingale with respect to the filtration $(\mathcal{F}_{t})_{t\geqslant0}$.
First, note that
\[
G_{t}(g):=\sigma\int_{0}^{t}\frac{\partial}{\partial x}g(s,X_{s})\sqrt{X_{s}}\mathrm{d}B_{s},\quad t\geqslant0,
\]
is a square-integrable martingale with respect to the filtration $(\mathcal{F}_{t})_{t\geqslant0}$.
Indeed, for each $t\geqslant0$, we have
\begin{align}
 & \mathbb{E}_{x}\left[\left(\sigma\int_{0}^{t}\sqrt{X_{s}}g_{x}'(s,X_{s})\mathrm{d}B_{s}\right)^{2}\right]\nonumber \\
 & \quad=\sigma^{2}\int_{0}^{t}e^{2cs}\mathbb{E}\left[\mathbbm{1}_{\lbrace X_{s}\leqslant1\rbrace}X_{s}V'(X_{s})\right]\mathrm{d}s+\sigma^{2}\kappa^{2}\int_{0}^{t}e^{2cs}\mathbb{E}\left[\mathbbm{1}_{\lbrace X_{s}>1\rbrace}X_{s}^{2\kappa-1}\right]\mathrm{d}s.\label{eq1: proof, lemma e ergod}
\end{align}
Clearly, we have $|\mathbbm{1}_{\lbrace X_{s}\leqslant1\rbrace}X_{s}V'(X_{s})|\leqslant\sup_{y\in[0,1]}|V'(y)|<\infty$,
which implies that the first integral on the right-hand side of \eqref{eq1: proof, lemma e ergod} is finite. Since $|\mathbbm{1}_{\lbrace X_{s}>1\rbrace}X_{s}^{2\kappa-1}|\leqslant|X_{s}|^{\kappa}$,
by \eqref{conditon: kappa moment} and Proposition \ref{prop:sup moments exist},
we see that the second integral on the right-hand side of \eqref{eq1: proof, lemma e ergod}
is finite as well. Hence, $(G_{t}(g))_{t\geqslant0}$ is a square-integrable martingale with respect to the filtration $(\mathcal{F}_{t})_{t\geqslant0}$.

Next, we prove that $J_{t}(g)$, $t\geqslant0$, is a martingale with
respect to the filtration $(\mathcal{F}_{t})_{t\geqslant0}$. We define
\begin{align*}
J_{\ast,t}(V) & :=\int_{0}^{t}\int_{\lbrace z\leqslant1\rbrace}e^{cs}\left(V(X_{s-}+z)-V(X_{s-})\right)\widetilde{N}(\mathrm{d}s,\mathrm{d}z),\quad t\geqslant0,\\
J_{t}^{\ast}(V) & :=\int_{0}^{t}\int_{\lbrace z>1\rbrace}e^{cs}\left(V(X_{s-}+z)-V(X_{s-})\right)\widetilde{N}(\mathrm{d}s,\mathrm{d}z),\quad t\geqslant0.
\end{align*}
So $J_{t}(g)=J_{\ast,t}(V)+J_{t}^{\ast}(V)$ for $t\geqslant0$. In
what follows, we establish some elementary inequalities for $V$.
For $y\geqslant1$, we have
\begin{align}
\mathbbm{1}_{\lbrace z\leqslant1\rbrace}(z)|V(y+z)-V(y)| & =\mathbbm{1}_{\lbrace z\leqslant1\rbrace}(z)\left((y+z)^{\kappa}-y^{\kappa}\right)\nonumber \\
 & =\mathbbm{1}_{\lbrace z\leqslant1\rbrace}(z)y^{\kappa}\left(\left(1+\frac{z}{y}\right)^{\kappa}-1\right)\nonumber \\
 & \leqslant\mathbbm{1}_{\lbrace z\leqslant1\rbrace}(z)\kappa y^{\kappa-1}z\leqslant\mathbbm{1}_{\lbrace z\leqslant1\rbrace}(z)z,\label{eq:first elementary ineq}
\end{align}
where we used Bernoulli's inequality to obtain the first inequality
in \eqref{eq:first elementary ineq}. Moreover, it is easy to see
that for $y\geqslant1$,
\begin{equation}
\mathbbm{1}_{\lbrace z>1\rbrace}(z)|V(y+z)-V(y)|\leqslant\mathbbm{1}_{\lbrace z>1\rbrace}(z)\left(y^{\kappa}+z^{\kappa}-y^{\kappa}\right)\leqslant\mathbbm{1}_{\lbrace z>1\rbrace}(z)z^{\kappa}.\label{eq:second elementary ineq}
\end{equation}
For $y\in[0,1]$, using the mean value theorem, we get
\begin{equation}
\mathbbm{1}_{\lbrace z\leqslant1\rbrace}(z)|V(y+z)-V(y)|\leqslant z\sup_{y\in[0,2]}|V'(y)|\leqslant c_{1}z,\label{eq:third elementary ineq}
\end{equation}
for some constant $c_{1}>0$. Finally, for $y\in[0,1]$, again by
Bernoulli's inequality, we have
\begin{align}
\mathbbm{1}_{\lbrace z>1\rbrace}(z)|V(y+z)-V(y)| & \leqslant\mathbbm{1}_{\lbrace z>1\rbrace}(z)\left((y+z)^{\kappa}+|V(y)|\right)\nonumber \\
 & \leqslant\mathbbm{1}_{\lbrace z>1\rbrace}(z)\left(z^{\kappa}\left(1+\kappa\frac{y}{z}\right)+|V(y)|\right)\nonumber \\
 & \leqslant\mathbbm{1}_{\lbrace z>1\rbrace}(z)\left(z^{\kappa}+1+|V(y)|\right)\nonumber \\
 & \leqslant\mathbbm{1}_{\lbrace z>1\rbrace}(z)\left(z^{\kappa}+c_{2}\right),\label{eq:fourth elementary ineq}
\end{align}
where $c_{2}:=1+\sup_{y\in[0,1]}|V(y)|<\infty$ is a positive constant.
Now, from \eqref{eq:first elementary ineq} and \eqref{eq:third elementary ineq},
we deduce that
\begin{align*}
\mathbb{E}_{x} & \left[\int_{0}^{t}\int_{\lbrace z\leqslant1\rbrace}e^{cs}|V(X_{s-}+z)-V(X_{s-})|\nu(\mathrm{d}z)\mathrm{d}s\right]\\
 & \quad\leqslant(1+c_{1})\int_{0}^{t}e^{cs}\mathrm{d}s\int_{\lbrace z\leqslant1\rbrace}z\nu(\mathrm{d}z)<\infty,\quad t\geqslant0.
\end{align*}
It follows from \cite[p.62 and Lemma 3.1]{MR1011252} that $(J_{\ast,t}(V))_{t\geqslant0}$
is a martingale with respect to the filtration $(\mathcal{F}_{t})_{t\geqslant0}$.
Using \eqref{eq:second elementary ineq} and \eqref{eq:fourth elementary ineq},
we obtain
\begin{align*}
\mathbb{E}_{x} & \left[\int_{0}^{t}\int_{\lbrace z>1\rbrace}e^{cs}|V(X_{s-}+z)-V(X_{s-})|\nu(\mathrm{d}z)\mathrm{d}s\right]\\
 & \quad\leqslant\int_{0}^{t}\int_{\lbrace z>1\rbrace}e^{cs}\left(z^{\kappa}+c_{2}\right)\nu(\mathrm{d}s)\mathrm{d}s\\
 & \quad=\int_{0}^{t}e^{cs}\mathrm{d}s\left(\int_{\lbrace z>1\rbrace}z^{\kappa}\nu(\mathrm{d}z)+c_{2}\nu(\lbrace z>1\rbrace)\right)<\infty,\quad t\geqslant0.
\end{align*}
As a consequence, we see that $(J_{t}^{\ast}(V))_{t\geqslant0}$ is
also a martingale. Clearly, $(M_{t}(g))_{t\geqslant0}=(G_{t}(g)+J_{t}(g))_{t\geqslant0}$
is now a martingale with respect to the filtration $(\mathcal{F}_{t})_{t\geqslant0}$.\\

``\emph{Step 2}'': We determine the constant $c\in\mathbb{R}_{>0}$
and find another positive constant $M<\infty$ such that
\begin{equation}
(\mathcal{A}V)(y)=(\mathcal{D}V)(y)+(\mathcal{J}V)(y)\leqslant-cV(y)+M,\quad y\in\mathbb{R}_{\geqslant0}.\label{eq:Foster-Lyapunov 2}
\end{equation}
Consider the jump part $\mathcal{J}V=\mathcal{J}_{\ast}V+\mathcal{J}^{\ast}V$,
where $\mathcal{J}_{\ast}V$ and $\mathcal{J}^{\ast}V$ are defined
by \eqref{eq:small jump part} and \eqref{eq:big jump part}, respectively.
For all $x\in\mathbb{R}_{\geqslant0}$, using \eqref{eq:first elementary ineq}
and \eqref{eq:third elementary ineq}, we obtain
\[
(\mathcal{J}_{\ast}V)(y)=\int_{\lbrace z\leqslant1\rbrace}|V(y+z)-V(y)|\nu(\mathrm{d}z)\leqslant(1+c_{1})\int_{\lbrace z\leqslant1\rbrace}z\nu(\mathrm{d}z)<\infty.
\]
For $\mathcal{J}^{\ast}V$, we can use \eqref{eq:second elementary ineq}
and \eqref{eq:fourth elementary ineq} to obtain that for all $y\in\mathbb{R}_{\geqslant0}$,
\begin{align*}
(\mathcal{J}^{\ast}V)(x) & =\int_{\lbrace z>1\rbrace}|V(y+z)-V(y)|\nu(\mathrm{d}z)\\
 & \leqslant\int_{\lbrace z>1\rbrace}z^{\kappa}\nu(\mathrm{d}z)+c_{2}\nu(\lbrace z>1\rbrace)<\infty.
\end{align*}
Next, we estimate $\mathcal{D}V$. Since,
\[
V'(x)=\kappa x^{\kappa-1}\quad\text{and}\quad V''(x)=\kappa(\kappa-1)x^{\kappa-2}\quad\text{for  }x\geqslant1,
\]
we see that
\begin{align*}
(\mathcal{D}V)(x) & =(a-bx)V'(x)+\frac{\sigma^{2}x}{2}V''(x)\\
 & =-b\kappa x^{\kappa}+\kappa x^{\kappa-1}\left(a+\frac{\sigma^{2}(\kappa-1)}{2}\right)\leqslant-b\kappa x^{\kappa}+c_{3}
\end{align*}
for all $x\geqslant1$. Here $c_{3}<\infty$ is a positive constant.
After all we get that for all $x\geqslant1$,
\[
(\mathcal{A}V)(x)\leqslant-b\kappa V(x)+c_{4}
\]
where $c_{4}<\infty$ is a positive constant. By noting that $V\in C^{2}(\mathbb{R}_{\geqslant0},\mathbb{R})$,
we see that
\[
\sup_{y\in[0,1]}|V(y)|<\infty\quad\mathrm{and\quad}\sup_{y\in[0,1]}|(\mathcal{A}V)(y)|<\infty.
\]
Consequently, \eqref{eq:Foster-Lyapunov 2} holds for all $x\geqslant0$.\\

``\emph{Step 3}'': We prove \eqref{eq:Foster-Lyapunov 1}. Note
that $(\mathcal{L}g)(s,x)=\exp\lbrace cs\rbrace(\mathcal{A}V)(x)$.
By \eqref{eq:ito formula to g}, \eqref{eq:Foster-Lyapunov 2} and
the martingale property of $(M_{t}(g))_{t\geqslant0}$, we obtain
that for all $(x,t)\in\mathbb{R}_{\geqslant0}^{2}$,
\begin{align*}
e^{ct}\mathbb{E}_{x}\left[V(X_{t})\right]-V(x) & =\mathbb{E}_{x}\left[g(t,X_{t})-g(0,X_{0})\right]\\
 & =\mathbb{E}_{x}\left[\int_{0}^{t}\left(e^{cs}(\mathcal{A}V)(X_{s})+ce^{cs}V(X_{s})\right)\mathrm{d}s\right]\\
 & \leqslant\mathbb{E}_{x}\left[\int_{0}^{t}\left(e^{cs}\left(-cV(X_{s})+M\right)+ce^{cs}V(X_{s})\right)\mathrm{d}s\right]\\
 & =\mathbb{E}_{x}\left[\int_{0}^{t}e^{cs}M\mathrm{d}s\right]\leqslant\frac{M}{c}e^{ct}.
\end{align*}
So \eqref{eq:Foster-Lyapunov 1} is true. With this our proof is complete.
\end{proof}

Based on Proposition \ref{lem:foster_lyapunov_fct}, we are now ready
to prove Theorem \ref{thm:ergodicity_of_y_x} (b).\\

\textit{Proof of Theorem} \ref{thm:ergodicity_of_y_x} (b).
In view of Proposition \ref{prop:existence and positivity of the density function}
and Proposition \ref{lem:foster_lyapunov_fct}, to obtain the exponential
ergodicity of $X$, we can follow almost the very same lines as in
the proof of \cite[Theorem 1]{MR3451177}. We remark that the \emph{strong}
\emph{aperiodicity} condition used in the proof of \cite[Theorem 1]{MR3451177}
can be safely replaced by the \emph{aperiodicity} condition (the definition
of aperiodicity can be found in \cite[p.114]{MR2509253}), due to \cite[Theorem 6.3]{MR1174380}. Moreover, the \emph{aperiodicity}
of the skeleton chain can be obtained by following the same arguments
as in part (b) of the proof of \cite[Theorem 6.1]{2016arXiv160706254J}.
This completes the proof.\qed \\

\textbf{Acknowledgements.} The author J. Kremer would like to thank the University of Wuppertal for the financial support through a doctoral funding program.

\def\cprime{$'$}
\providecommand{\bysame}{\leavevmode\hbox to3em{\hrulefill}\thinspace}
\providecommand{\MR}{\relax\ifhmode\unskip\space\fi MR }
\providecommand{\MRhref}[2]{%
  \href{http://www.ams.org/mathscinet-getitem?mr=#1}{#2}
}
\providecommand{\href}[2]{#2}

\end{document}